\documentclass[12pt, a4paper]{article}
\usepackage{latexsym}
\usepackage{graphicx} 
\usepackage{amsmath} 
\usepackage{amsfonts} 
\usepackage{amsthm} 
\usepackage{amssymb}   
\usepackage{enumerate}

\usepackage{authblk}

\long\def\delete#1{}

\newcommand{\diam}{\mathsf{diam}}  
\newcommand{\dist}{\mathsf{dist}}

\newtheorem{theorem}{Theorem}[section]
\newtheorem{lemma}[theorem]{Lemma} 
\newtheorem{corollary}[theorem]{Corollary}

\newtheorem{problem}[theorem]{Problem}

\def\qed{\hfill$\Box$\vspace{12pt}}

\def\x{\mathbf x}
\def\y{\mathbf y}
\def\b{\mathbf b}

\def\b0{{\bf 0}}

\def\b{\beta} 
 
\def\l{\lambda} 
\def\nl{\overline{\lambda}} 
\def\s{\sigma} 
\def\ns{\overline{\sigma}}

\def\sp{{\rm sp}} 
\def\mod{{\rm mod}} 

\usepackage{xcolor}
\usepackage[normalem]{ulem}

\date{}

\begin{document}

\title{Distance-constrained labellings of Cartesian products of graphs}
  
\author[1]{Anna Llad\'{o}\thanks{E-mail: \texttt{aina.llado@upc.edu}}}
\author[2]{Hamid Mokhtar\thanks{E-mail: \texttt{hamid.mokhtar@alumni.unimelb.edu.au}}}
\author[1]{Oriol Serra\thanks{E-mail: \texttt{oriol.serra@upc.edu}}}
\author[2]{Sanming Zhou\thanks{E-mail: \texttt{sanming@unimelb.edu.au}}}

\affil[1]{\small Department of Mathematics, Universitat Polit\`ecnica de Catalunya,  Barcelona, Spain}
\affil[2]{\small School of Mathematics and Statistics, The University of Melbourne, Parkville, VIC 3010, Australia}
 
\openup 0.8 \jot 
\maketitle

\begin{abstract}
An $L(h_1, h_2, \ldots, h_l)$-labelling of a graph $G$ is a mapping $\phi: V(G) \rightarrow \{0, 1, 2, \ldots\}$ such that for $1\le i\le l$ and each pair of vertices $u, v$ of $G$ at distance $i$, we have  $|\phi(u) - \phi(v)| \geq h_i$. The span of $\phi$ is the difference between the largest and smallest labels assigned to the vertices of $G$ by $\phi$, and $\l_{h_1, h_2, \ldots, h_l}(G)$ is defined as the minimum span over all $L(h_1, h_2, \ldots, h_l)$-labellings of $G$. 

In this paper we study $\l_{h, 1, \ldots, 1}$ for Cartesian products of graphs, where $(h, 1, \ldots, 1)$ is an $l$-tuple with $l \ge 3$. We prove that, under certain natural conditions, the value of this and three related invariants on a graph $H$ which is the Cartesian product of $l$ graphs attain a common lower bound. In particular, the chromatic number of the $l$-th power of $H$ equals this lower bound plus one. We further obtain a sandwhich theorem which extends the result to a family of subgraphs of $H$ which contain a certain subgraph of $H$. All these results apply in particular to the class of Hamming graphs: if $q_1\ge \cdots \ge q_d\ge 2$ and $3\le l\le d$ then the Hamming graph $H=H_{q_1,q_2,\ldots ,q_d}$ satisfies  $\lambda_{q_l,1,\ldots,1}(H) = q_1q_2\ldots q_l-1$ whenever $q_1q_2\ldots q_{l-1}>3(q_{l-1}+1)q_l\ldots q_d$. In particular, this settles a case of the open problem on the chromatic number of powers of the hypercubes.


\end{abstract}

\section{Introduction}
\label{sec:introd}

Motivated by the frequency assignment problem \cite{Gamst, Hale} for communication networks, various optimal labelling problems for graphs involving distance conditions have been studied extensively since the 1980s. Among them is the following well-known distance labelling problem: Given a graph $G$ and nonnegative integers $h_1, h_2, \ldots, h_l$, determine the smallest positive integer $k$ with the property that each vertex of $G$ can be assigned a label from $\{0, 1, \ldots, k\}$ such that for $1 \le i \le l$ every pair of vertices at distance $i$ receive labels which differ by at least $h_i$. The $\l_{h_1, h_2, \ldots, h_l}$-number of $G$, denoted by $\l_{h_1, h_2, \ldots, h_l}(G)$, is defined as the smallest positive integer $k$ with this property. This parameter and several variants of it, especially for small $l$, have received much attention in the past more than three decades. In particular, a number of results on the $\l_{h_1, h_2}$-number have been produced by many researchers, especially in the case when $(h_1, h_2) = (2, 1)$, as one can find in the survey paper \cite{Cala}. Much work in this case was motivated by a conjecture of Griggs and Yeh \cite{GY} which asserts that $\l_{2, 1}(G) \le \Delta^2 $ for any graph $G$ with maximum degree $\Delta$. As far as we know, this conjecture is still open in its general form, though it has been confirmed in many special cases (see, for example, \cite{GM, GY, HRS, Kang, Sakai}). In recent years, the $\l_{h_1, h_2, h_3}$-number has also received considerable attention (see, for example, \cite{CFTV, FGKLP, KRZ, KLZ, Zhou-1}), but for $l > 3$ very little is known about the $\l_{h_1, h_2, \ldots, h_l}$-number. In general, it is difficult to determine the exact value of $\l_{h_1, h_2, \ldots, h_l}$ for a general graph. For example, for $l = 1$, we have $\l_{h_1}(G) = h_{1}(\chi(G)-1)$, and so determining $\l_{h_1}$ is equivalent to computing the chromatic number $\chi$. Answering a question posed in \cite{KRZ}, it was proved in \cite{FGKLP} that the problem of determining the $\l_{h_1, 1, 1}$-number is NP-complete even for trees. 

In this paper we study the $\l_{h, 1, \ldots, 1}$-number and three variants of it (see Subsection \ref{subsec:4p} for their definitions) for any graph $H$ which is the Cartesian product of $l$ non-trivial graphs, where $l \ge 3$ and $(h, 1, \ldots, 1)$ is an $l$-tuple with $h \ge 1$. We prove that under a certain condition these four invariants for $H$ all attain a common lower bound, and in particular the chromatic number of the $l$-th power of $H$ is equal to this lower bound plus $1$. We obtain further a sandwich theorem which says that under the same condition the same result holds for every subgraph of $H$ that contains a certain subgraph of $H$ as a subgraph. As corollaries we obtain that these results are true for Hamming graphs $H_{q_1, q_2, \ldots, q_d}$ such that $q_1 q_2 \ldots q_{l-1} > 3(\min\{q_1, \ldots, q_{l-1}\}+1)q_{l} \ldots q_{d}$ for some $l$ with $3 \le l < d$. We will give the precise statements of our results in Theorems \ref{thm:cart}--\ref{thm:sand-cart} and Corollaries \ref{cor:Lh21}--\ref{cor:Lhqd} after introducing relevant definitions and giving a brief review of related results in Subsections \ref{subsec:4p} and \ref{subsec:hm}, respectively. Our results give infinite families of graphs for which the values of $\l_{h, 1, \ldots, 1}$ and three variants of it can be computed exactly.

\subsection{Some basic terminology}

All graphs considered in the paper are finite, undirected and simple. Let $G$ be a graph with vertex set $V(G)$ and edge set $E(G)$. As usual denote by $\chi(G)$ the chromatic number of $G$ and call $|V(G)|$ the order of $G$. Denote by $\dist_G(u, v)$ the distance in $G$ between vertices $u, v$ of $G$. For an integer $l \ge 1$, the \emph{$l$-th power} $G^l$ of $G$ is the graph with vertex set $V(G)$ in which $u, v \in V(G)$ are adjacent if and only if $1 \le \dist_G(u,v) \le l$. We write $K \subseteq G$ to denote that $K$ is a subgraph of $G$.

The \emph{Cartesian product} of given graphs $G_1, G_2, \ldots, G_d$, denoted by $G_1 \Box G_2 \Box \cdots \Box G_d$, is the graph with vertex set $V(G_1) \times V(G_2) \times \cdots \times V(G_d)$ in which two vertices $(u_1, u_2, \ldots, u_d)$, $(v_1, v_2, \ldots, v_d)$ are adjacent if and only if there is exactly one $i \in \{1, 2, \ldots, d\}$ such that $u_i v_i \in E(G_i)$ and $u_j = v_j$ for all $j \in \{1, 2, \ldots, d\} \setminus \{i\}$. 

Given integers $q_1, q_2, \ldots, q_d \ge 2$, the \emph{Hamming graph} $H_{q_1, q_2, \ldots, q_d}$ is the Cartesian product $K_{q_1} \Box K_{q_2} \Box \cdots \Box K_{q_d}$ where, for an integer $q \ge 1$, $K_q$ denotes the complete graph with order $q$. Since the Cartesian product is commutative, without loss of generality we may assume that $q_1 \geq q_2 \ldots \geq q_d$. In the case when $q_1 = q_2 = \cdots = q_d = q$, we write $H(d,q)$ in place of $H_{q_1, q_2, \ldots, q_d}$. In particular, $H(d, 2)$ is the $d$-dimensional \emph{hypercube} $Q_d$.

\subsection{Distance-constrained labelling problems}
\label{subsec:4p}

Let $h_1, h_2, \ldots, h_l$ be nonnegative integers. An {\em $L(h_1, h_2, \ldots, h_l)$-labelling} of $G$ is a mapping $\phi$ from $V(G)$ to the set of nonnegative integers such that, for $i = 1, 2, \ldots, l$ and any pair of vertices $u, v \in V(G)$ with $\dist_G(u, v) = i$,  
\begin{equation} 
\label{eq:lc} |\phi(u) - \phi(v)| \geq h_i. 
\end{equation} 
The integer $\phi(u)$ is the \emph{label} of $u$ under $\phi$ and the \emph{span} of $\phi$, denoted by $\sp(G; \phi)$, is the difference between the largest and smallest labels assigned to the vertices of $G$ by $\phi$. Without loss of generality we may always assume that the smallest label used is $0$, so that
$$
\sp(G; \phi) = \max_{v \in V(G)} \phi(v).
$$ 
The \textit{$\l_{h_1, h_2, \ldots, h_l}$-number} of $G$ is defined \cite{GM-1, GY} as 
$$
\l_{h_1, h_2, \ldots, h_l}(G) = \min_{\phi} \sp(G; \phi),
$$ 
where the minimum is taken over all $L(h_1, h_2, \ldots, h_l)$-labellings of $G$. Equivalently, as stated in the beginning of this paper, $\l_{h_1, h_2, \ldots, h_l}(G)$ is the smallest positive integer $k$ such that an $L(h_1, h_2, \ldots, h_l)$-labelling of $G$ with span $k$ exists. 

The above notion of distance labelling originated from the frequency assignment problem \cite{Hale} for which the value $\l_{h_1, h_2, \ldots, h_l}(G)$ measures the minimum bandwidth required by a radio communication network modelled by $G$ under the constraints \eqref{eq:lc}. It is readily seen that  
$$
\chi(G^l) = \l_{1,1,\ldots,1}(G) + 1,
$$
where $(1, 1, \ldots, 1)$ is an $l$-tuple. Thus, from a pure graph-theoretical point of view, the $L(h_1, h_2, \ldots, h_l)$-labelling problem can be considered as a generalization of the classical vertex-colouring problem.  

An $L(h_1, h_2, \ldots, h_l)$-labelling $\phi$ of $G$ is said to be 
\emph{no-hole} (see, for example, \cite{CLZ, CLZ-1, R, SW, S}) if $\{\phi(v): v \in V(G)\}$ is a set of consecutive integers. Define $\nl_{h_1, h_2, \ldots, h_l}(G)$ to be the minimum span among all no-hole $L(h_1, h_2, \ldots, h_l)$-labellings of $G$, and $\infty$ if no such a labelling exists. As an example, we see that $\nl_{h_1,h_2,h_3}(K_q) = \infty$ for all $h_1 \ge 2$ and $q \ge 2$. 

The $L(h_1, h_2, \ldots, h_l)$-labelling problem and its no-hole version are a linear model in the sense that the $L_1$-metric is used to measure the span between two channels. The cyclic version of the $L(h_1, h_2, \ldots, h_l)$-labelling problem was studied in \cite{HLS} with a focus on small $l$. A mapping $\phi: V(G) \rightarrow \{0, 1, 2, \cdots, k-1\}$ is called a \emph{$C(h_1, h_2, \ldots, h_l)$-labelling} of $G$ with {\em span $k$} if, for $i = 1, 2, \ldots, l$ and any $u, v \in V(G)$ with $\dist_G(u, v) = i$,  
$$
|\phi(u) - \phi(v)|_{k} \ge h_i, 
$$
where 
$$
|x-y|_{k} = \min\{|x-y|, k-|x-y|\}
$$ 
is the {\em $k$-cyclic distance} between $x$ and $y$. A $C(h_1, h_2, \ldots, h_l)$-labelling of $G$ with span $k$ exists for sufficiently large $k$. The {\em $\s_{h_1, h_2, \ldots, h_l}$-number} of $G$, denoted by $\s_{h_1, h_2, \ldots, h_l}(G)$, is defined to be the minimum integer $k-1$ such that $G$ admits a $C(h_1, h_2, \ldots, h_l)$-labelling with span $k$. Note that $\s_{h_1, h_2, \ldots, h_l}(G)$ thus defined agrees with $\s(G; h_1, h_2, \ldots, h_l)$ defined in \cite{CLZ} but is one less than $\s(G; h_1, h_2, \ldots, h_l)$ used in \cite{HLS} and $c_{h_1, h_2, \ldots, h_l}(G)$ used in \cite{L}. As observed in \cite{Gamst, HLS}, this cyclic version allows the assignment of a set of channels $\phi(u), \phi(u)+k, \phi(u)+2k, \ldots$ to each transmitter $u$ when $G$ is viewed as a radio network with one transmitter placed at each vertex. 

A $C(h_1, h_2, \ldots, h_l)$-labelling $\phi$ of $G$ with span $k$ is \emph{no-hole} if $\{\phi(v): v \in V(G)\}$ is a set of consecutive integers $\mod~k$. Define $\ns_{h_1, h_2, \ldots, h_l}(G)$ to be the minimum $k-1$ such that $G$ admits a no-hole $C(h_1, h_2, \ldots, h_l)$-labelling of span $k$, and $\infty$ if no such a labelling exists. 

It can be verified that if $h_1 \ge h_2 \ge \cdots \ge h_l$ then the four invariants above are all \emph{monotonically increasing}; that is, $ \eta(H) \le \eta(G) $ for $\eta = \l_{h_1, h_2, \ldots, h_l}, \nl_{h_1, h_2, \ldots, h_l},  \s_{h_1, h_2, \ldots, h_l}, \ns_{h_1, h_2, \ldots, h_l}$ whenever $H$ is a subgraph of $G$.

\subsection{Distance-constrained labellings of Hamming graphs}
\label{subsec:hm}

This paper was motivated by distance-constrained labellings of Hamming graphs and hypercubes. As such let us mention several known results on this class of graphs. More results can be found in the short survey \cite{Zhou-2}.

In \cite[Theorem 3.7]{WGM} it was proved that, if $2^{n-1} \leq d \leq 2^n - t$ for some $t$ between $1$ and $n+1$, then 
\begin{equation}
\label{eq:21Qd}
\l_{2,1}(Q_d) \leq 2^n + 2^{n-t+1} - 2.
\end{equation}
In \cite[Theorem 3.1]{GMS} it was shown that, if $p$ is a prime and either $d \leq p$ and $r \geq 2$, or $d < p$ and $r = 1$, then
\begin{equation}
\label{eq:gms}
\l_{2,1}(H(d, p^r)) = p^{2r} - 1.
\end{equation} 
The $\l_{j,k}$-number of $H_{q_1, q_2}$ was determined in \cite{GMS} and results on $H_{q_1, q_2, q_3}$ can be found in \cite{GM,LLS}. 

In \cite{Zhou} a group-theoretic approach to $L(j,k)$-labelling Cayley graphs of Abelian groups was introduced. As an application it was proved \cite{Zhou} among other things that
$$
\l_{j,k}(H_{q_1, q_2, \ldots, q_d}) = (q_1 q_2 - 1)k
$$ 
for any $2k \ge j \ge k \ge 1$ if $q_1 > d \ge 2$, $q_2$ divides $q_1$ and is no less than $q_3, \ldots, q_d$, and every prime factor of $q_1$ is no less than $d$, generalizing \eqref{eq:gms} to a wide extent. In \cite{Zhou} it was also proved (as a corollary of a more general result) that $\l_{j,k}(Q_d) \leq 2^n \max\{k, \lceil j/2
\rceil\} + 2^{n-t}\min\{j-k, \lfloor j/2 \rfloor\}-j$, which yields 
\begin{equation}
\label{eq:jkQd}
\l_{j,k}(Q_d) \leq 2^{n} k + 2^{n -t} (j-k)-j
\end{equation}
when $2k \ge j$, where $n = 1 + \lfloor \log_{2} d \rfloor$ and $t = \min\{2^n - d - 1, n\}$. In the special case when $(j,k) = (2,1)$, the upper bound \eqref{eq:jkQd} gives exactly \eqref{eq:21Qd} (see \cite[p.990]{Zhou} for justification). In \cite{Zhou-1} lower and upper bounds on $\l_{h_1,h_2,h_3}(Q_d)$ were obtained using a group-theoretic approach, which recover the main result in \cite{KDP} in the special case when $(h_1,h_2,h_3) = (1,1,1)$. The problem of determining $\l_{1, \ldots, 1}(Q_d)$, or equivalently the chromatic number of powers of $Q_d$, has a long history but is still wide open. See \cite{KK, Zhou-1} for some background information and related results. One of the contributions of the present paper settles this problem for a range of dimensions (see Corollary \ref{cor:Lhqd}). 

Note that $\l_{j,k}(H_{q_1, q_2, \ldots, q_d}) \ge (q_1 q_2 - 1)k$ for $j \ge k$ as $H_{q_1, q_2, \ldots, q_d}$ contains $H_{q_1, q_2}$ as a subgraph and $H_{q_1, q_2}$ has order $q_1 q_2$ and diameter two. The following question was asked in \cite[Question 6.1]{Zhou} (see also \cite[Section 5]{CLZ-1}): Given integers $j$ and $k$ with $2k \ge j \ge k \ge 1$, for which integers $q_1 \ge q_2 \ge \cdots \ge q_d$ with $j/k \le q_1q_2 - \sum_{i=1}^{d}q_i + d$ do we have $\l_{j,k}(H_{q_1, q_2, \ldots, q_d}) = (q_1 q_2 - 1)k$? A partial answer to this question was given in \cite[Theorem 1.3]{CLZ-1}, where it was proved that, for $(j,k)=(2,1), (1,1)$, if $q_1$ is sufficiently large, namely $q_1 \ge d+n-1+ \sum_{i=2}^d (i-2) (q_i - 1)$, where $n$ is the largest subscript such that $q_2 = q_n$, then 
$$
\l_{j, k}(H_{q_1, q_2, \ldots, q_d}) = \nl_{j, k}(H_{q_1, q_2, \ldots, q_d}) = \ns_{j, k}(H_{q_1, q_2, \ldots, q_d}) = \s_{j, k}(H_{q_1, q_2, \ldots, q_d}) = q_1 q_2 - 1.
$$ 
This result inspired us to explore when a similar phenomenon occurs for $\lambda_{h,1,\ldots ,1}$, $\nl_{h,1,\ldots ,1}$, $\ns_{h,1,\ldots ,1}$ and $\sigma_{h,1,\ldots ,1}$ for Cartesian products of graphs. As will be seen in the next subsection, our main results in the present paper provide sufficient conditions for this to happen.

\subsection{Main results}
\label{subsec:main}

The first main result in this paper is as follows.  

\begin{theorem}
\label{thm:cart} 
Let $G_1, \ldots, G_{l-1}$ and $G$ be non-trivial graphs with orders $q_1, \ldots, q_{l-1}$ and $q$, respectively, and let $H = G_1 \Box \cdots \Box G_{l-1} \Box G$, where $l \ge 3$. Let $q_{l}$ be an integer with $1 \le q_l \le q$. Suppose that
$$
q_1 q_2 \ldots q_{l-1} > 3(\min\{q_1, \ldots, q_{l-1}\}+1)q
$$
and $H$ contains a subgraph $K$ with order $q_1 q_2 \ldots q_{l}$ and diameter at most $l$. Then for any integer $h$ with $1 \le h \le q_{l}$ we have
\begin{equation}
\label{eq:hH}
\lambda_{h,1,\ldots ,1} (H) = \nl_{h,1,\ldots ,1} (H) = \ns_{h,1,\ldots ,1} (H) = \sigma_{h,1,\ldots ,1} (H) = q_1 q_2 \ldots q_{l}-1,
\end{equation}
where $(h,1,\ldots,1)$ is an $l$-tuple. Moreover, there is a labelling of $H$ that is optimal for $\l_{h, 1, \ldots, 1}$, $\nl_{h, 1, \ldots, 1}$, $\ns_{h, 1, \ldots, 1}$ and $\s_{h, 1, \ldots, 1}$ simultaneously. In particular, we have
$$
\chi(H^l) = q_1 q_2 \ldots q_l
$$ 
and the same labelling gives rise to an optimal colouring of   $H^l$. 
\end{theorem}

Since by our assumption $H$ contains a subgraph with order $q_1 q_2 \ldots q_{l}$ and diameter at most $l$, it is easy to see that $q_1 q_2 \ldots q_{l}-1$ is a lower bound for each of the four invariants in \eqref{eq:hH}. Theorem \ref{thm:cart} asserts that actually these four invariants for $H$ all achieve this trivial lower bound. In Section \ref{sec:rem}, we will give a general construction to show that there are many graphs other than Hamming graphs which satisfy the conditions of Theorem \ref{thm:cart}.

Using Theorem \ref{thm:cart}, we obtain the following sandwich result, which is our second main result in the paper. The claimed optimal labelling in this sandwich theorem is the restriction of the above mentioned optimal labelling of $H$ to $V(X)$. Recall that we write $K \subseteq G$ when $K$ is a subgraph of a graph $G$.

\begin{theorem}[Sandwich Theorem]  
\label{thm:sand-cart} 
Under the conditions of Theorem \ref{thm:cart}, for every graph $X$ with $K \subseteq X \subseteq H$ and any integer $h$ with $1 \le h \le q_{l}$, we have
\begin{equation}
\label{eq:hX}
\lambda_{h,1,\ldots ,1} (X) = \nl_{h,1,\ldots ,1} (X) = \ns_{h,1,\ldots ,1} (X) = \sigma_{h,1,\ldots ,1} (X) = q_1 q_2 \ldots q_{l}-1,
\end{equation}
where $(h,1,\ldots,1)$ is an $l$-tuple. Moreover, there is a labelling of $X$ that is optimal for $\l_{h, 1, \ldots, 1}$, $\nl_{h, 1, \ldots, 1}$, $\ns_{h, 1, \ldots, 1}$ and $\s_{h, 1, \ldots, 1}$ simultaneously. In particular, we have
$$
\chi(X^l) = q_1 q_2 \ldots q_l
$$ 
and the same labelling gives rise to an optimal colouring of $X^l$. 
\end{theorem}

Setting $G_1 = K_{q_1}, \ldots, G_{l-1} = K_{q_{l-1}}$ and $G = K_{q_l} \Box \ldots \Box K_{q_d}$ in Theorem \ref{thm:cart}, we have $H = G_1 \Box \cdots \Box G_{l-1} \Box G = H_{q_1, q_2, \ldots, q_d}$. Since $H_{q_1, q_2, \ldots, q_d}$ has subgraph $H_{q_1, q_2, \ldots, q_l}$ with order $q_1 q_2 \ldots q_{l}$ and diameter $l$, Theorem \ref{thm:cart} implies immediately the following result for Hamming graphs.  
  
\begin{corollary}
\label{cor:Lh21}
Let $q_1\ge q_2\ge \cdots\ge q_d$ be integers no less than $2$, and let $l$ be an integer with $3 \le l < d$. Let $H = H_{q_1, q_2, \ldots, q_d}$. Suppose that 
$$
q_1 q_2 \ldots q_{l-1} > 3(q_{l-1}+1)q_{l} \ldots q_{d}.
$$ 
Then for any integer $h$ with $1 \le h \le q_l$ we have
\begin{equation}
\label{eq:hamming}
\l_{h,1, \ldots, 1}(H) = \nl_{h,1, \ldots, 1}(H) = \ns_{h,1, \ldots, 1}(H) = \s_{h,1, \ldots, 1}(H) = q_1 q_2 \ldots q_l - 1, 
\end{equation}
where $(h,1,\ldots ,1)$ is an $l$-tuple. Moreover, there is a labelling of $H$ that is optimal for $\l_{h,1, \ldots, 1}$, $\nl_{h,1, \ldots, 1}$, $\ns_{h,1, \ldots, 1}$ and $\s_{h,1, \ldots, 1}$ simultaneously. In particular, we have
$$
\chi(H^l) = q_1 q_2 \ldots q_l
$$ 
and the same labelling gives rise to an optimal colouring of $H^l$. 
\end{corollary}  

Similarly, Theorem \ref{thm:sand-cart} implies the following result, in which the claimed optimal labelling is the restriction of the above mentioned optimal labelling of $H_{q_1, q_2, \ldots, q_d}$ to $V(X)$. 

\begin{corollary}[Sandwich Theorem for Hamming Graphs]  
\label{cor:sand-gen}
Under the conditions of Corollary \ref{cor:Lh21}, for every graph $X$ such that $H_{q_1, q_2, \ldots, q_l}\subseteq X \subseteq H_{q_1, q_2, \ldots, q_d}$ and any integer $h$ with $1 \le h \le q_l$, we have   
$$
\l_{h,1,\ldots,1}(X) = \nl_{h,1,\ldots,1}(X) = \ns_{h,1,\ldots,1}(X) = \s_{h,1,\ldots,1}(X) = q_1 q_2 \ldots q_l - 1,
$$
where $(h,1,\ldots ,1)$ is an $l$-tuple. Moreover, there is a labelling of $X$ that is optimal for $\l_{h,1,\ldots,1}$, $\nl_{h,1,\ldots,1}$, $\ns_{h,1,\ldots,1}$ and $\s_{h,1,\ldots,1}$ simultaneously. In particular, we have
$$
\chi(X^l) = q_1 q_2 \ldots q_l
$$ 
and the same labelling gives rise to an optimal colouring of $X^l$.   
\end{corollary}

Corollary \ref{cor:Lh21} implies the following result for Hamming graphs $H(d,q)$ and the $d$-dimensional hypercube $Q_d = H(d, 2)$.

\begin{corollary}
\label{cor:Lhqd}
Let $d, q$ and $l$ be integers such that $d \ge 6$, $q \ge 2$ and 
\begin{equation}
\label{eq:dql}
(d + 4 + \max\{4-q, 0\})/2 \le l < d.
\end{equation}
Then for any integer $h$ with $1 \le h \le q$ we have
$$
\l_{h,1,\ldots ,1} (H(d,q))=q^{l}-1,
$$
where $(h, 1, \ldots, 1)$ is an $l$-tuple. In particular, if $(d+6)/2 \le l < d$, then for $h = 1, 2$, 
\begin{equation}
\label{eq:Qd}
\l_{h,1,\ldots ,1} (Q_d)=2^{l}-1.
\end{equation}
\end{corollary}

Note that \eqref{eq:Qd} requires $d \ge 8$. This is so because for $q = 2$ the inequalities in \eqref{eq:dql} cannot be true unless $d \ge 8$. Similarly, for $q = 3$, \eqref{eq:dql} requires $d \ge  7$. In the general case when $q \ge 4$, \eqref{eq:dql} says that $l$ is between $(d+4)/2$ and $d-1$. 

Theorems \ref{thm:cart} and \ref{thm:sand-cart} will be proved in Section \ref{sec:cart} after a short preparation in the next section. The paper concludes in Section \ref{sec:rem} with a construction illustrating the wide applicability of Theorems \ref{thm:cart} and \ref{thm:sand-cart}, some final remarks assessing the strength of the sufficient conditions in Theorem \ref{thm:cart} and two open problems.

\section{Preliminaries}
\label{sec:prel}

The following inequalities follow immediately from related definitions. 

\begin{lemma}
\label{lem:chain} 
Let $G$ be a graph and let $h_1 \ge h_2 \ge \cdots \ge h_l$ be nonnegative integers. Then 
\begin{equation} 
\label{eq:ls} 
\l_{h_1, h_2, \ldots, h_l}(G) \le \s_{h_1, h_2, \ldots, h_l}(G) \le \l_{h_1, h_2, \ldots, h_l}(G) + h_1 - 1~\mbox{(\cite{HLS})} 
\end{equation}
\begin{equation} 
\label{eq:nls} \l_{h_1, h_2, \ldots, h_l}(G) \le \nl_{h_1, h_2, \ldots, h_l}(G) \le \ns_{h_1, h_2, \ldots, h_l}(G) 
\end{equation} 
\begin{equation} 
\label{eq:nls-20} \s_{h_1, h_2, \ldots, h_l}(G) \le \ns_{h_1, h_2, \ldots, h_l}(G). 
\end{equation} 
\end{lemma} 

\begin{corollary}
\label{cor:chain}
Let $G$ be a graph and let $h_1 \ge h_2 \ge \cdots \ge h_l$ be nonnegative integers. If $\l_{h_1, h_2, \ldots, h_l}(G)$ $=$ $\ns_{h_1, h_2, \ldots, h_l}(G)$, then 
$$
\l_{h_1, h_2, \ldots, h_l}(G) = \nl_{h_1, h_2, \ldots, h_l}(G) = \ns_{h_1, h_2, \ldots, h_l}(G) = \s_{h_1, h_2, \ldots, h_l}(G)
$$ 
and any optimal no-hole $C(h_1, h_2, \ldots, h_l)$-labelling of $G$ is optimal for $\l_{h_1, h_2, \ldots, h_l}$, $\nl_{h_1, h_2, \ldots, h_l}$, $\ns_{h_1, h_2, \ldots, h_l}$ and $\s_{h_1, h_2, \ldots, h_l}$ simultaneously. 
\end{corollary} 

\begin{proof} 
Since $\l_{h_1, h_2, \ldots, h_l}(G) = \ns_{h_1, h_2, \ldots, h_l}(G)$, by (\ref{eq:nls}) we have $\l_{h_1, h_2, \ldots, h_l}(G) = \nl_{h_1, h_2, \ldots, h_l}(G) = \ns_{h_1, h_2, \ldots, h_l}(G)$. Thus the first inequality in (\ref{eq:ls}) becomes $\ns_{h_1, h_2, \ldots, h_l}(G) \le \s_{h_1, h_2, \ldots, h_l}(G)$. This together with (\ref{eq:nls-20}) implies $\ns_{h_1, h_2, \ldots, h_l}(G) = \s_{h_1, h_2, \ldots, h_l}(G)$. Obviously, the statement about optimality holds. 
\qed
\end{proof} 

\begin{lemma} 
\label{lem:ham-cart}  
Let $G_1, \ldots, G_{l-1}$ and $G$ be non-trivial graphs with orders $q_1, \ldots, q_{l-1}$ and $q$, respectively, where $l \ge 2$. Let $H = G_1 \Box \cdots \Box G_{l-1} \Box G$. Let $h_1 \ge h_2 \ge \cdots \ge h_l$ be positive integers, and let $q_l$ be an integer with $1 \le q_l \le q$. If $\ns_{h_1, h_2, \ldots, h_l}(H) \le q_1 q_2 \cdots q_l - 1$ and $H$ contains a subgraph with order $q_1 q_2 \ldots q_{l}$ and diameter at most $l$, then  
$$
\l_{h_1, h_2, \ldots, h_l}(H) = \nl_{h_1, h_2, \ldots, h_l}(H) = \ns_{h_1, h_2, \ldots, h_l}(H) = \s_{h_1, h_2, \ldots, h_l}(H) =  q_1 q_2 \ldots q_l - 1
$$ 
and any optimal no-hole $C(h_1, h_2, \ldots, h_l)$-labelling of $H$ is optimal for $\l_{h_1, h_2, \ldots, h_l}$, $\nl_{h_1, h_2, \ldots, h_l}$, $\s_{h_1, h_2, \ldots, h_l}$ and $\ns_{h_1, h_2, \ldots, h_l}$ simultaneously. 
\end{lemma} 

\begin{proof}
By our assumption, $H$ contains a subgraph $K$ with order $q_1 q_2 \ldots q_{l}$ and diameter at most $l$. The vertices of $K$ must receive distinct labels under any $L(h_1, h_2, \ldots, h_l)$-labelling of $H$. Hence 
$$
\l_{h_1, h_2, \ldots, h_l}(H) \ge q_1 q_2 \ldots q_l - 1. 
$$
Since $\ns_{h_1, h_2, \ldots, h_l}(H) \le q_1 q_2 \ldots q_l - 1$ by our assumption, by (\ref{eq:nls}) we then have $\l_{h_1, h_2, \ldots, h_l}(H) = \ns_{h_1, h_2, \ldots, h_l}(H) = q_1 q_2 \ldots q_l - 1$. The result now follows from Corollary \ref{cor:chain} immediately. 
\qed 
\end{proof}

\section{Proofs of Theorems \ref{thm:cart} and \ref{thm:sand-cart}}
\label{sec:cart}

\subsection{A lemma}

\begin{lemma}
\label{lem:cart} 
Let $G_1, \ldots, G_{l-1}$ and $G$ be non-trivial graphs with orders $q_1, \ldots, q_{l-1}$ and $q$, respectively, and let $H = G_1 \Box \cdots \Box G_{l-1} \Box G$, where $l \ge 3$. If 
\begin{equation}
\label{eq:3q}
q_1 q_2 \ldots q_{l-1} > 3(\min\{q_1, \ldots, q_{l-1}\}+1)q, 
\end{equation}
then for any integer $q_{l}$ with $1 \le q_{l} \le q$, we have
$$
\ns_{q_{l},1,\ldots,1} (H)\le q_1 q_2 \ldots q_{l}-1,
$$
where $(q_{l},1,\ldots,1)$ is an $l$-tuple.  
\end{lemma}

\begin{proof} 
Since $G_1, \ldots, G_{l-1}$ and $G$ are non-trivial graphs, their orders $q_1, \ldots, q_{l-1}$ and $q$ are no less than $2$. Since the Cartesian product is commutative, without loss of generality we may assume that $q_{l-1} = \min\{q_1, \ldots, q_{l-1}\}$, so \eqref{eq:3q} becomes $q_1 q_2 \ldots q_{l-1} > 3(q_{l-1}+1)q$. Denote the vertices of $G_i$ as 
$$
V(G_i)=\{0,1, \ldots, q_i-1\},\, 1\le i\le l-1
$$ 
and the vertices of $G$ as 
$$
V(G)=\{0,1,\ldots,q-1\}.
$$ 
Then 
$$
V(H) = \{(x_1, \ldots, x_{l-1},x): 0 \le x_i \le q_i-1 \text{ for } 1\le i\le l-1, 0 \le x \le q-1\}. 
$$
Let $q_{l}$ be an integer with $1 \le q_{l} \le q$. Set  
$$
N_i=\prod_{j=i}^{l}q_j, \, 1\le i\le l,
$$ 
and 
$$
N_{l+1}=1.
$$  
Obviously, for any $t$ with $1 \le t \le l$, every integer in the interval $[0, q_1 q_2 \ldots q_{t}-1]$ can be uniquely expressed as 
\begin{equation}
\label{eq:x1}
x_1 q_{2} \ldots q_{t} + x_2 q_{3} \ldots q_{t} + \cdots + x_{t-1} q_{t} + x_t,\; 0\le x_i\le q_i-1 \text{ for } 1 \le i \le t.
\end{equation}
Conversely, any integer of this form is in $[0, q_1 \ldots q_{t} - 1]$. This establishes a bijection between the integers in $[0, q_1 q_2 \ldots q_{t} - 1]$ and the vectors $(x_1,\ldots,x_{t})$ of integers with $0\le x_i\le q_i-1$ for $1 \le i \le t$.  In particular, every integer in $[0,N_1-1]$ can be uniquely written as 
\begin{equation}
\label{eq:x}
\sum_{i=1}^{l} x_iN_{i+1},\; 0\le x_i\le q_i-1 \text{ for } 1 \le i \le l
\end{equation}
and conversely any integer of this form is in $[0,N_1-1]$. In this way we establish a bijection between the integers in $[0,N_1-1]$ and the vectors $(x_1,\ldots,x_{l})$ of integers with $0\le x_i\le q_i-1$ for $1 \le i \le l$. 

For each $x\in V(G)$, we define 
$$
r(x) \equiv x \pmod {q_{l}}
$$
to be the unique integer in $\{0,1, \ldots, q_{l}-1\}$ congruent to $x$ modulo $q_{l}$. The uniqueness of $r(x)$ is due to the assumption that $q_{l} \le q$. Since $q < 3(q_{l-1}+1)q < q_1 q_2 \ldots q_{l-1} \le N_1$, every integer $x \in V(G)$ can be expressed in the form of \eqref{eq:x} and moreover $r(x) = x_{l}$. 

For each integer $t$ between $0$ and $q-1$, consider a set
\begin{equation}
\label{eq:A}
A_t = \{(a_1(x), \ldots, a_{l-1}(x)): 0 \le x \le t\}
\end{equation}
of vectors $(a_1(x), \ldots, a_{l-1}(x))$ of integers such that $0\le a_{i}(x)\le q_i-1$ for $1 \le i \le l-1$ and $0 \le x \le t$. For $(x_1, \ldots, x_{l-1}, x) \in V(H)$ with $0 \le x \le t$, define
$$
\phi_{A_t} (x_1,\ldots, x_{l-1},x)=\sum_{i=1}^{l-1}\big((a_{i}(x)+x_i)~\mod~{q_{i}}\big)N_{i+1} + r(x).
$$
The most important ingredient of this proof is the following statement.

\smallskip
\textsc{Claim 1:} For each integer $t$ between $0$ and $q-1$, there exists a set $A_t$ of vectors as in \eqref{eq:A} such that, for any pair of distinct vertices $\x = (x_1, \ldots, x_{l-1},x)$, $\y = (y_1, \ldots, y_{l-1}, y)$ of $H$ with $0 \le x,y \le t$, we have
\begin{equation}\label{eq:a1}
|\phi_{A_t}(\x)-\phi_{A_t} (\y)|_{N_1}\ge q_{l} \; \text{ if } \; \dist_{H}(\x,\y)=1
\end{equation}
and
\begin{equation}\label{eq:a2}
|\phi_{A_t}(\x)-\phi_{A_t}(\y)|_{N_1}\ge 1 \; \text{ if } \; 1 < \dist_{H}(\x,\y) \le l.
\end{equation}

In fact, once this is proved, we then obtain that $\phi_{A_{q-1}}$ is a no-hole $C(q_{l}, 1, \ldots, 1)$-labelling of $H$ with span $N_1$ and hence $\ns_{q_{l},1,\ldots,1} (H) \le N_1 - 1 = q_1 q_2 \ldots q_{l}-1$ as desired. 

We prove Claim 1 by induction on $t$. In the case when $t = 0$, we set $(a_1(0), \ldots, a_{l-1}(0))$ $= (0, \ldots, 0)$ so that $A_0 = \{(0, \ldots, 0)\}$. In this case we have $x=y=0$ and so $r(x)=r(y)=0$. Thus, for $\x \ne \y$, $\phi_{A_0}(\x) = \sum_{i=1}^{l-1} x_i N_{i+1}$ and $\phi_{A_0} (\y) = \sum_{i=1}^{l-1} y_i N_{i+1}$ are distinct multiples of $q_{l}$, which implies that $|\phi_{A_0} (\x)-\phi_{A_0} (\y)|_{N_1}\ge q_{l}$. Therefore, both \eqref{eq:a1} and \eqref{eq:a2} are satisfied by distinct vertices $\x, \y$ with $x=y=0$.

Let $0 < t < q-1$. Assume that the statement in Claim 1 holds for nonnegative integers smaller than $t$. So in particular the existence of $A_{t-1}$ is assumed. We will prove that there exists a vector $(a_1(t), \ldots, a_{l-1}(t))$ of integers with $0 \le a_{i}(t) \le q_i-1$ for each $i$ such that if we set 
\begin{equation}
\label{eq:at}
A_{t} = A_{t-1} \cup \{(a_1(t), \ldots, a_{l-1}(t))\}
\end{equation}
then conditions \eqref{eq:a1} and \eqref{eq:a2} are satisfied by $A_{t}$ and any pair of distinct vertices $\x=(x_1, \ldots, x_{l-1}, x)$, $\y=(y_1, \ldots, y_{l-1}, y)$ of $H$ with $0 \le x,y \le t$. Without loss of generality we may assume that $y \le x$. If $x \le t-1$, then for any choice of $(a_1(t), \ldots, a_{l-1}(t))$ the set $A_{t}$ given in \eqref{eq:at} satisfies $\phi_{A_{t}}(\x) = \phi_{A_{t-1}}(\x)$ and $\phi_{A_{t}}(\y) = \phi_{A_{t-1}}(\y)$. Thus, by our hypothesis, any pair of distinct vertices $\x, \y$ with $0 \le y\le x \le t-1$ satisfies \eqref{eq:a1} and \eqref{eq:a2}, regardless of the choice of $(a_1(t), \ldots, a_{l-1}(t))$. So in what follows we only consider pairs of distinct vertices $\x , \y$ of $H$ with $0 \le y\le x = t$.  

\smallskip
\textsc{Case 1.} $(x_1,\ldots,x_{l-1})= (y_1,\ldots,y_{l-1})$.  
\smallskip

In this case we have $0 \le y < x = t$ and hence the vector $(a_1(y), \ldots, a_{l-1}(y)) \in A_{t-1}$ has been defined already by our hypothesis. Set
$$
\psi_y(z_1,\ldots, z_{l-1}) = \sum_{i=1}^{l-1} \big((z_i-a_{i} (y))~\mod~{q_i}\big)N_{i+1},
$$
where $0 \le z_i \le q_i - 1$ for each $i$. Then $\psi_y(z_1,\ldots, z_{l-1})$ is a multiple of $q_{l}$. Just as \eqref{eq:x1} defines a bijection, one can see that $\psi_y/q_{l}$ is a bijection from the set of vectors $(z_1,\ldots, z_{l-1})$ to the integer interval $[0, (N_1/q_{l}) - 1]$. If we set $A_t = A_{t-1} \cup \{(z_1,\ldots, z_{l-1})\}$, then 
$$
\phi_{A_t} (\x)-\phi_{A_t} (\y)=\psi_y(z_1,\ldots, z_{l-1}) + r(x) - r(y).
$$
Since $|r(x) - r(y)|_{N_1}<q_{l}$, if $(z_1,\ldots, z_{l-1})$ is chosen in such a way that
\begin{equation}\label{eq:psi1}
|\psi_y(z_1,\ldots, z_{l-1})|_{N_1}\ge 2q_{l},
\end{equation} 
then conditions \eqref{eq:a1} and \eqref{eq:a2} are satisfied by $A_t$ and all pairs of distinct vertices $\x , \y \in V(H)$ with $(x_1,\ldots,x_{l-1})= (y_1,\ldots,y_{l-1})$ and $0 \le y < x = t$. There are $N_1/q_{l}$ choices for $(z_1,\ldots, z_{l-1})$ which give rise to pairwise distinct integer values of $\psi_y(z_1,\ldots z_{l-1})/q_{l}$ ranging from $0$ to $(N_1/q_{l}) - 1$. So for each $y$ there are three choices for $(z_1,\ldots z_{l-1})$ which violate \eqref{eq:psi1}, namely when $\psi_y (z_1,\ldots z_{l-1})/q_{l}$ takes values $0$, $1$ or $(N_1/q_{l}) - 1$. Since $y$ ranges from $0$ to $t-1$, we see that in total there are at most $3t$ ($< 3q$) choices for $(z_1,\ldots z_{l-1})$ which violate \eqref{eq:psi1} for some $y$. Therefore, there are at most $3q$ choices for $(z_1,\ldots z_{l-1})$ such that \eqref{eq:a1} or \eqref{eq:a2} is violated for some pair of distinct $\x, \y$ with $(x_1,\ldots,x_{l-1})= (y_1,\ldots,y_{l-1})$ and $0 \le y < x = t$. 

\smallskip
\textsc{Case 2.} $(x_1,\ldots ,x_{l-1})\neq (y_1,\ldots ,y_{l-1})$.   
\smallskip

Consider $A_t = A_{t-1} \cup \{(z_1,\ldots, z_{l-1})\}$, where $0 \le z_i \le q_i - 1$ for each $i$. If $x=y$, then
$$
\phi_{A_t} (\x)-\phi_{A_t} (\y) = \sum_{i=1}^{l-1} \big((x_i-y_i)~\mod~{q_i}\big) N_{i+1}.
$$
By \eqref{eq:x1}, $\phi_{A_t} (\x)-\phi_{A_t} (\y)$ is a multiple of $q_{l}$ and is $q_{l}$ apart from $0$ and $N_1$. So conditions \eqref{eq:a1} and \eqref{eq:a2} are satisfied by $A_t$ and the pair $\x, \y$, regardless of the choice of $(z_1,\ldots,z_{l-1})$. 

Now suppose that $y<x=t$. (Recall that we assumed $x = t$ at the end of the paragraph containing \eqref{eq:at}.) Then $d(\x,\y)\ge 2$ and condition \eqref{eq:a1} is not required for $\x$ and $\y$. We have
\begin{eqnarray}
\phi_{A_t} (\x)-\phi_{A_t} (\y)  
& = & \left\{\sum_{i=1}^{l-1} \big((z_i+x_i)~\mod~{q_i}-(a_i(y)+y_i)~\mod~{q_i}\big) q_{i+1} \ldots q_{l-1}\right\} q_{l} \nonumber \\
& & +\ r(x)-r(y). \label{eq:disa}
\end{eqnarray}
Note that $|r(x)-r(y)| \le q_{l} - 1$. Thus, if $((z_1 + x_1)~\mod~{q_1}, \ldots, (z_{l-1} + x_{l-1}) ~\mod~{q_{l-1}})$ disagrees with $((a_{1}(y)+y_1)~\mod~{q_1}, \ldots, (a_{l-1}(y)+y_{l-1})~\mod~{q_{l-1}})$ in at least two coordinates or in the $i$-th coordinate only for some $1 \le i \le l-2$, then the absolute value of the first term on the right hand side of \eqref{eq:disa} is no less than $2q_{l}$ and so condition \eqref{eq:a1} is satisfied by $\x$ and $\y$. The same statement holds if these two vectors disagree in the $(l-1)$-th coordinate only but $|(z_{l-1} + x_{l-1})~\mod~{q_{l-1}} - (a_{l-1}(y)+y_{l-1})~\mod~{q_{l-1}}| \ge 2$. On the other hand, for a fixed $y$, there are at most three choices for $z_{l-1}$ such that $|(z_{l-1} + x_{l-1})~\mod~{q_{l-1}} - (a_{l-1}(y)+y_{l-1})~\mod~{q_{l-1}}| \le 1$. Therefore,  for a fixed $\x$ and all $y \le t$, there are at most $3t$ ($\le 3q$) choices for $(z_1,\ldots, z_{l-1})$ such that \eqref{eq:a2} is violated. Since $x_{l-1}$ ranges from $0$ to $q_{l-1}-1$,  there are at most $3q_{l-1} q$ choices for $(z_1,\ldots, z_{l-1})$ such that \eqref{eq:a2} is violated by some pair $\x, \y$ with $(x_1,\ldots ,x_{l-1})\neq (y_1,\ldots ,y_{l-1})$ and $y < x = t$.

In summary, we have proved that there are at most $3q + 3q_{l-1} q$ choices for $(z_1,\ldots, z_{l-1})$ such that \eqref{eq:a1} or \eqref{eq:a2} is violated by some pair of distinct vertices $\x, \y$ of $H$. Since $N_1/q_{l} = q_1 q_2 \ldots q_{l-1} > 3(q_{l-1}+1)q$ by our assumption, among the $N_1/q_{l}$ choices for $(z_1,\ldots, z_{l-1})$ there exists at least one which can be set as $(a_1(t), \ldots, a_{l-1}(t))$ such that \eqref{eq:a1} and \eqref{eq:a2} are satisfied by $A_{t} = A_{t-1} \cup \{(a_1(t), \ldots, a_{l-1}(t))\}$ and all pairs of distinct vertices $\x=(x_1, \ldots, x_{l-1}, x)$, $\y=(y_1, \ldots, y_{l-1}, y)$ of $H$ with $0 \le x,y \le t$. By mathematical induction, we have proved Claim 1 and hence the lemma. 
\qed
\end{proof}

\subsection{Proofs of Theorems \ref{thm:cart} and \ref{thm:sand-cart}}
    
\begin{proof}{\textbf{of Theorem \ref{thm:cart}}}~~      
In the special case when $h = q_l$, the result follows from Lemmas \ref{lem:ham-cart} and \ref{lem:cart} immediately. Thus all equalities in \eqref{eq:hH} hold when $h = q_l$. 

Consider any integer $h$ with $1 \le h \le q_l$. By our assumption, $H$ contains a subgraph with order $q_1q_2 \ldots q_l$ and diameter at most $l$. All vertices of this subgraph should be assigned pairwise distinct labels under any $L(h, 1, \ldots, 1)$-labelling of $H$. So we have $\l_{h, 1, \ldots, 1}(H) \ge q_1q_2 \ldots q_l - 1$. On the other hand, $\l_{h, 1, \ldots, 1}(H) \le \l_{q_{l}, 1, \ldots, 1}(H) = q_1q_2 \ldots q_l - 1$ as $1 \le h \le q_l$. Therefore, $\l_{h, 1, \ldots, 1}(H) = q_1q_2 \ldots q_l - 1$. Since $h \le q_l$, we have $\ns_{h,1,\ldots,1}(H) \le \ns_{q_l,1,\ldots,1}(H) = q_1q_2 \ldots q_l - 1$. Combining this with \eqref{eq:ls} and \eqref{eq:nls-20}, we obtain $q_1q_2 \ldots q_l - 1 = \l_{h, 1, \ldots, 1}(H) \le \s_{h, 1, \ldots, 1}(H) \le \ns_{h, 1, \ldots, 1}(H) \le \ns_{q_{l}, 1, \ldots, 1}(H) = q_1q_2 \ldots q_l - 1$. Hence $\l_{h, 1, \ldots, 1}(H) = \ns_{h, 1, \ldots, 1}(H) = q_1q_2 \ldots q_l - 1$. It then follows from Corollary \ref{cor:chain} that all equalities in \eqref{eq:hH} hold for $h$ and any optimal no-hole $C(h, 1, \ldots, 1)$-labelling of $H$ is optimal for $\l_{h, 1, \ldots, 1}, \nl_{h, 1, \ldots, 1}, \ns_{h, 1, \ldots, 1}$ and $\s_{h, 1, \ldots, 1}$ simultaneously.
\qed 
\end{proof}

\begin{proof}{\textbf{of Theorem \ref{thm:sand-cart}}}~~      
Since $K \subseteq X \subseteq H$, for $\eta = \l_{h,1,\ldots,1}$ or $\s_{h,1,\ldots,1}$, we have $\eta(K) \le \eta(X) \le \eta(H)$ as $\eta$ is monotonically increasing. Moreover, $\eta(H) = q_1 q_2 \ldots q_l - 1$ by Theorem \ref{thm:cart}. Since $K$ has order $q_1 q_2 \ldots q_l$ and diameter at most $l$, we have $\l_{h,1,\ldots,1}(K) \ge q_1 q_2 \ldots q_l - 1$. This together with \eqref{eq:ls} implies that $\eta(K) \ge q_1 q_2 \ldots q_l - 1$ for $\eta = \l_{h,1,\ldots,1}$ or $\s_{h,1,\ldots,1}$. Combining this with $\eta(K) \le \eta(X) \le \eta(H) = q_1 q_2 \ldots q_l - 1$, we obtain that $\eta(X) = q_1 q_2 \ldots q_l - 1$ for $\eta = \l_{h,1,\ldots,1}$ or $\s_{h,1,\ldots,1}$. 

By Theorem \ref{thm:cart}, any optimal no-hole $L(h,1,\ldots,1)$- or $C(h,1,\ldots,1)$-labelling $\phi$ of $H$ has span $q_1 q_2 \ldots q_l - 1$. Since $K$ is a subgraph of $H$ with order $q_1 q_2 \ldots q_l$ and diameter at most $l$, all labels used by $\phi$ must appear in $K$. Since $K \subseteq X$, it follows that the restriction of $\phi$ to $V(X)$ is a no-hole $L(h,1,\ldots,1)$- or $C(h,1,\ldots,1)$-labelling of $X$. Thus $\eta(X) \le \eta(H) = q_1q_2 \ldots q_l - 1$ for $\eta = \nl_{h,1,\ldots,1}$ or $\ns_{h,1,\ldots,1}$. Similarly, we have $\eta(K) \le \eta(X)$ as $K \subseteq X$ and $\eta(X) \le q_1q_2 \ldots q_l - 1$. On the other hand, since $K$ has diameter at most $l$, we have $\eta(K) \ge q_1 q_2 \ldots q_l - 1$. Therefore, $\eta(X) = q_1q_2 \ldots q_l - 1$ for $\eta = \nl_{h,1,\ldots,1}$ or $\ns_{h,1,\ldots,1}$. Hence we have proved that all equalities in \eqref{eq:hX} hold. Moreover, one can see that the restriction to $V(X)$ of the optimal labelling in Theorem \ref{thm:cart} is a labelling of $X$ that is optimal for $\l_{h,1,\ldots,1}, \nl_{h,1,\ldots,1}, \ns_{h,1,\ldots,1}$ and $\s_{h,1,\ldots,1}$ simultaneously.

Since $K$ has order $q_1q_2 \ldots q_l$ and diameter at most $l$, we have $K_{q_1q_2 \ldots q_l} \cong K^l \subseteq X^l \subseteq H^l$. Since $\chi(K_{q_1q_2 \ldots q_l}) = q_1q_2 \ldots q_l$, and $\chi(H^l) = q_1q_2 \ldots q_l$ by Theorem \ref{thm:cart}, it follows that $\chi(X^l) = q_1q_2 \ldots q_l$ and the same labelling as above gives rise to an optimal colouring of $X^l$.
\qed 
\end{proof}

\section{Concluding remarks}
\label{sec:rem}

It is not difficult to construct many graphs other than Hamming graphs which satisfy the conditions of Theorem \ref{thm:cart}, and we give a simple construction here. Let $q_1, \ldots, q_{l-1}$ be integers no less than $2$ such that $q_{l-1} = \min\{q_1, \ldots, q_{l-1}\}$ and $3q_{l-1}(q_{l-1}+1) < q_1$, where $l \ge 3$. Let $q_l$ be an integer between $1$ and $q_{l-1}$. Then $q_2 \ldots q_{l-1}q_{l} < (q_1 q_2 \ldots q_{l-2} q_{l})/(3(q_{l-1}+1))$. Take an integer $q$ such that $q_2 \ldots q_{l-1}q_{l} \le q < (q_1 q_2 \ldots q_{l-2} q_{l})/(3(q_{l-1}+1))$. Then $1 \le q_{l} \le q_{l-1} < q$ and $3(q_{l-1}+1)q < q_1 q_2 \ldots q_{l-2} q_{l} \le q_1 q_2 \ldots q_{l-2} q_{l-1}$. Let $G_1 = K_{q_1}$, and let $G_2, \ldots, G_{l-1}$ be graphs with orders $q_2, \ldots, q_{l-1}$, respectively. Let $G$ be a graph with order $q$ which contains a subgraph $G^*$ with order $q_2 \ldots q_{l-1}q_{l}$ and diameter $l-1$. (There are many graphs $G$ satisfying these conditions.) Let $H = G_1 \Box \cdots \Box G_{l-1} \Box G$ and $K = G_1 \Box G^*$. Then $K$ is a subgraph of $H$ with order $q_1 q_2 \ldots q_{l-1}q_{l}$ and diameter $\diam(K) = \diam(K_{q_1}) + \diam(G^*) = 1 + (l-1) = l$.
So all conditions in Theorem \ref{thm:cart} are satisfied but $H$ is not necessarily a Hamming graph. 

A question related to the main contributions in this paper is to assess the strength of the sufficient condition on the sizes of the factors of Cartesian products of graphs in Theorem \ref{thm:cart} and its corollaries. In the case of the Hamming graph $H_{q_1, q_2, \ldots, q_d}$, there are exactly $\sum_{i=1}^{l_0} \left(\sum_{S \subseteq \{1, \ldots, d\}, |S|=i}\; \prod_{j \in S} (q_j-1)\right) + 1$ vertices of $H_{q_1, q_2, \ldots, q_d}$ at distance no more than $l_0 = \lfloor l/2 \rfloor$ from a fixed vertex, and these vertices require pairwise distinct labels in any $L(h,1, \ldots, 1)$-labelling. Thus, for the conclusion of Corollary \ref{cor:Lh21} to hold, a necessary condition is
$$
q_1 q_2 \ldots q_l \ge \sum_{i=1}^{l_0} \left(\sum_{S \subseteq \{1, \ldots, d\}, |S|=i}\; \prod_{j \in S} (q_j-1)\right).
$$
Moreover, since in any $L(h,1, \ldots, 1)$-labelling the $\sum_{i=1}^d (q_i-1)$ neighbours of the $0$-labelled vertex should receive pairwise distinct labels no less than $h$, for the conclusion of Corollary \ref{cor:Lh21} to hold, another necessary condition is
$$
q_1 q_2 \ldots q_l \ge h + \sum_{i=1}^d (q_i-1).
$$
Both conditions are met under the assumptions of Corollary \ref{cor:Lh21}. However, it is not clear whether these two obvious necessary conditions are also sufficient for the conclusion of Corollary \ref{cor:Lh21}. 

In this dirction it would be interesting to study the following problems (see \cite[Question 6.1]{Zhou} and \cite[Question 5.1]{CLZ-1} for two related questions for distance-$2$ labellings of Hamming graphs): 

\begin{problem}
Let $q_1\ge q_2\ge \cdots\ge q_d$ be integers no less than $2$, and let $l$ be an integer with $3 \le l \le d$. Let $(h,1,\ldots ,1)$ be an $l$-tuple with $h$ a positive integer. 
\begin{itemize}
\item[\rm (a)] Give necessary and sufficient conditions for $\l_{h,1, \ldots, 1}(H_{q_1, q_2, \ldots, q_d}) = \nl_{h,1, \ldots, 1}(H_{q_1, q_2, \ldots, q_d}) = \ns_{h,1, \ldots, 1}(H_{q_1, q_2, \ldots, q_d}) = \s_{h,1, \ldots, 1}(H_{q_1, q_2, \ldots, q_d}) = q_1 q_2 \ldots q_l - 1$ to hold. 
\item[\rm (b)] Give necessary and sufficient conditions for $\l_{h,1, \ldots, 1}(H_{q_1, q_2, \ldots, q_d}) = q_1 q_2 \ldots q_l - 1$ to hold.
\end{itemize}
\end{problem}

\section*{Acknowledgements}

We would like to thank the anonymous referees for their helpful comments. Anna Llad\'o and Oriol Serra acknowledge financial support from the Spanish {\it Agencia Estatal de Investigaci\'on} under project  MTM2017-82166-P. Zhou was supported by the Research Grant Support Scheme of The University of Melbourne.
     
\small { 

} 

\end{document}